\documentclass[12pt,leqno,final]{article}
\usepackage{amsmath, verbatim, color}
\usepackage{mathrsfs}
\usepackage{dsfont}
\usepackage[a4paper,margin=0.91in]{geometry}

\usepackage[notcite,notref]{showkeys}
\usepackage{amsmath, amsthm, amsfonts, amssymb, color}
\usepackage{mathrsfs}
\usepackage{color}
\usepackage{stmaryrd}
\usepackage[notcite,notref]{showkeys}
\makeatletter
\newcommand{\Spvek}[2][r]{%
  \gdef\@VORNE{1}
  \left(\hskip-\arraycolsep%
    \begin{array}{#1}\vekSp@lten{#2}\end{array}%
  \hskip-\arraycolsep\right)}

\def\vekSp@lten#1{\xvekSp@lten#1;vekL@stLine;}
\def\vekL@stLine{vekL@stLine}
\def\xvekSp@lten#1;{\def\temp{#1}%
  \ifx\temp\vekL@stLine
  \else
    \ifnum\@VORNE=1\gdef\@VORNE{0}
    \else\@arraycr\fi%
    #1%
    \expandafter\xvekSp@lten
  \fi}
\makeatother

\newtheorem{thm}{Theorem}[section]
\newtheorem{cor}[thm]{Corollary}
\newtheorem{lem}[thm]{Lemma}
\newtheorem{prp}[thm]{Proposition}

\newtheorem{exa}[thm]{Example}
\newtheorem{rem}[thm]{Remark}
\theoremstyle{definition}

\newcommand{\scr}[1]{\mathscr #1}
\definecolor{wco}{rgb}{0.5,0.2,0.3}

\numberwithin{equation}{section} \theoremstyle{remark}

\newcommand{\ua}{\uparrow}

\title{{\bf    Harnack Inequalities for Functional SDEs
Driven by Subordinate Brownian Motions}\footnote{Supported in
 part by  NNSFC (11801406, 11831015).}
}
\author{
{\bf     Chang-Song Deng $^{a)}$, Xing Huang$^{b)}$,    }\\
\footnotesize{  a)School of Mathematics and Statistics, Wuhan University, Wuhan 430072, China}\\
\footnotesize{ dengcs@whu.edu.cn }\\
\footnotesize{  b)Center for Applied Mathematics, Tianjin University, Tianjin 300072, China}\\
\footnotesize{  xinghuang@tju.edu.cn}}
\begin{document}
\allowdisplaybreaks
\def\R{\mathbb R}  \def\ff{\frac} \def\ss{\sqrt} \def\B{\mathbf
B}
\def\N{\mathbb N} \def\kk{\kappa} \def\m{{\bf m}}
\def\ee{\varepsilon}\def\ddd{D^*}
\def\dd{\delta} \def\DD{\Delta} \def\vv{\varepsilon} \def\rr{\rho}
\def\<{\langle} \def\>{\rangle} \def\GG{\Gamma} \def\gg{\gamma}
  \def\nn{\nabla} \def\pp{\partial} \def\E{\mathbb E}
\def\d{\text{\rm{d}}} \def\bb{\beta} \def\aa{\alpha} \def\D{\scr D}
  \def\si{\sigma} \def\ess{\text{\rm{ess}}}
\def\beg{\begin} \def\beq{\begin{equation}}  \def\F{\scr F}
\def\Ric{\text{\rm{Ric}}} \def\Hess{\text{\rm{Hess}}}
\def\e{\text{\rm{e}}} \def\ua{\underline a} \def\OO{\Omega}  \def\oo{\omega}
 \def\tt{\tilde} \def\Ric{\text{\rm{Ric}}}
\def\cut{\text{\rm{cut}}} \def\P{\mathbb P} \def\ifn{I_n(f^{\bigotimes n})}
\def\C{\scr C}   \def\G{\scr G}   \def\aaa{\mathbf{r}}     \def\r{r}
\def\gap{\text{\rm{gap}}} \def\prr{\pi_{{\bf m},\varrho}}  \def\r{\mathbf r}
\def\Z{\mathbb Z} \def\vrr{\varrho} \def\ll{\lambda}
\def\L{\scr L}\def\Tt{\tt} \def\TT{\tt}\def\II{\mathbb I}
\def\i{{\rm in}}\def\Sect{{\rm Sect}}  \def\H{\mathbb H}
\def\M{\scr M}\def\Q{\mathbb Q} \def\texto{\text{o}} \def\LL{\Lambda}
\def\Rank{{\rm Rank}} \def\B{\scr B} \def\i{{\rm i}} \def\HR{\hat{\R}^d}
\def\to{\rightarrow}\def\l{\ell}\def\iint{\int}
\def\EE{\scr E}\def\no{\nonumber}
\def\A{\scr A}\def\V{\mathbb V}\def\osc{{\rm osc}}
\def\BB{\scr B}\def\Ent{{\rm Ent}}\def\3{\triangle}\def\H{\scr H}
\def\U{\scr U}\def\8{\infty}\def\1{\lesssim}\def\HH{\mathrm{H}}
 \def\T{\scr T}
  \newcommand\I{\mathds 1}
\maketitle

\begin{abstract}  Using coupling by change of measure and an approximation technique, Wang's Harnack inequalities are established for a class of functional SDEs driven by subordinate Brownian motions. The results cover the corresponding ones in the case without delay.
\end{abstract} \noindent
 \vskip 2cm

\section{Introduction}

The dimension-free Harnack inequality was firstly introduced by Wang \cite{FYW0} to derive the log-Sobolev inequality on Riemannian manifolds. As a weaker version of the
power-Harnack inequality, the log-Harnack inequality
was considered in \cite{RW} for
semi-linear SDEs. These two Harnack-type inequalities have been intensively investigated and applied for various finite- and
infinite-dimensional SDEs
and SPDEs driven by Brownian noise; we refer to
the monograph by F.-Y.\ Wang \cite{Wbook} for a systematic theory on dimension-free Harnack inequalities and applications.
For the functional SDEs and SPDEs, the Harnack inequalities are also investigated in \cite{BWY,BWY13}, see also \cite{SWY} for SDEs with non-Lipschitz coefficients and \cite{HW, HZ} for SDEs with Dini drifts. However, the noise in all the above results is assumed to
contain a Brownian motion part. The central aim
of this work is to establish Harnack inequalities
for functional SDEs driven by subordinate Brownian
motions, which form a very large class of
L\'{e}vy processes. It turns out that our results
cover the corresponding ones in the case without delay
derived by J.\ Wang and F.-Y.\ Wang \cite{WW}
(cf.\ \cite{Den14} for an improved estimate).

Fix a constant $r_0\geq0$. Denote by $\C$ the family
of all right continuous functions
$f:[-r_0,0]\to\mathbb{R}^{d}$ with left limits. To characterize the state space, equip $\C$ with the norm
$\|\cdot\|_2$ given by
$$\|\xi\|_2^2:=\int_{-r_0}^0|\xi(s)|^2\,\d s+|\xi(0)|^2
,\quad \xi\in\C.
$$
For $f:[-r_0,\infty)\to\mathbb{R}^{d}$,
we will denote $f_t \in \C$, $t\geq 0$, the
corresponding segment process, by
letting
$$f_t(s):=f(t+s),\quad s\in [-r_0,0].$$

Let $S=(S(t))_{t\geq0}$ be a subordinator (without killing), i.e.\ a nondecreasing L\'{e}vy process on $[0,\infty)$ starting
at $S(0)=0$. Due to the independent and stationary increments property, it is uniquely determined by the Laplace
transform
$$
    \E\,\e^{-uS(t)}=\e^{-t\phi(u)},\quad u>0,\,t\geq 0,
$$
where the characteristic (Laplace)
exponent $\phi:(0,\infty)\rightarrow(0,\infty)$ is a Bernstein
function with $\phi(0+):=\lim_{r\downarrow0}\phi(r)=0$,
i.e.\ a $C^\infty$-function such
that $(-1)^{n-1}\phi^{(n)}\geq0$ for all $n\in\N$.
Every such $\phi$ has a unique L\'{e}vy--Khintchine representation (cf. \cite[Theorem 3.2]{SSV})
\begin{equation}\label{bern}
    \phi(u)
    =\kappa u+\int_{(0,\infty)}\left(1-\e^{-ux}\right) \,\nu(\d x),\quad u>0,
\end{equation}
where $\kappa\geq0$ is the drift parameter and
$\nu$ is a L\'{e}vy measure on $(0,\infty)$
satisfying $$\int_{(0,\infty)}(1\wedge x)
\,\nu(\d x)<\infty.$$ It is clear that
$\tilde{\phi}(u):=\phi(u)-\kappa u$ is the Bernstein
function of the subordinator
$\tilde{S}(t):=S(t)-\kappa t$ having zero
drift and L\'{e}vy measure $\nu$.

Consider the following functional SDEs on $\mathbb{R}^{d}$:
\beq\label{E1}
\d X(t)=b(X(t))\,\d t+B(X_t)\,\d t +\d W(S(t)),
\end{equation}
where $W=(W(t))_{t\geq 0}$ is a $d$-dimensional standard Brownian motion with respect to a complete filtered probability space $(\OO, \F, \{\F_{t}\}_{t\ge 0}, \P)$, $S=(S(t))_{t\geq 0}$ is a subordinator with Bernstein function of the
form \eqref{bern} and independent of $W$, $b: \mathbb{R}^{d}\to \mathbb{R}^{d}$ is continuous,
and $B: \C\to \mathbb{R}^{d}$ is measurable.

We shall need the following conditions
on $b$ and $B$:
\beg{enumerate}
\item[\bf{(H)}] There exist  constants $K\in\mathbb{R}$ and $K_1\geq0$ such that
$$
\langle x-y,b(x)-b(y)\rangle\leq K|x-y|^2, \quad x,y\in\mathbb{R}^d,
$$
and
$$
 |B(\xi)-B(\eta)|\leq K_1\|\xi-\eta\|_{2},\quad \xi,\eta\in\C.
$$
\end{enumerate}

\begin{rem}\label{EAU}
The condition {\bf{(H)}} ensures the existence, uniqueness
and non-explosion of the solution to \eqref{E1}. Indeed,
letting $L(t)=W(S(t))$, $\hat{b}(t,x)=b(x+L(t))$ and
$\hat{B}(t,\xi)=B(\xi+L_t)$, one has
$$\langle x-y,\hat{b}(t,x)-\hat{b}(t,y)\rangle\leq K|x-y|^2, \quad x,y\in\mathbb{R}^d,t\geq0$$
and
$$ |\hat{B}(t,\xi)-\hat{B}(t,\eta)|\leq K_1\|\xi-\eta\|_{2},\quad \xi,\eta\in\C,t\geq0.$$
Then the following (functional) ordinary differential
equation
$$\d \hat{X}(t)=\hat{b}(t,\hat{X}(t))\,\d t+\hat{B}(t,\hat{X}_t)\,\d t$$
has a unique solution which does not explode in finite
time; setting $X(t):=\hat{X}(t)+L(t)$, we know that
\eqref{E1} has a unique non-explosive solution.
\end{rem}

For $\xi\in\C$, let $X_t^\xi$ be the solution to \eqref{E1} with $X_0=\xi$. Let $P_t$ be the semigroup associated to $X_t^\xi$, i.e.
$$P_t f(\xi)=\mathbb{E}f(X_t^\xi),\quad f\in\B_b(\C).$$

The remaining part of this paper is organized as follows.
In Section 2, we state our main results. By using the coupling by change of measure and an approximation technique, we
establish in Section 3 the Harnack inequalities
for functional SDEs
driven by non-random time-changed Brownian
motions. Section 4 is devoted to
the proofs of Theorem \ref{T3.2}
and Example \ref{ex1} presented in Section 2.

\section{Main results}

As usual, we make the convention that
$\frac10=\infty$ and $0\cdot\infty=0$.

\begin{thm}\label{T3.2} Assume
  {\bf (H)} and let $T>r_0$ and $S$ be a subordinator
  with Bernstein function $\phi$ of the form \eqref{bern}.

  \smallskip\noindent\textup{i)} \
  For any $\xi, \eta\in \C$ and
$f\in \B_b(\C)$ with $f\geq1$,
\beg{equation*}\beg{split}
P_T\log f(\eta)&\leq\log P_T f(\xi)+|\xi(0)-\eta(0)|^2\,
\mathbb{E}\left(\int_0^{T-r_0}\e^{-2Kt}\,\d S(t)\right)^{-1}
\\
&\quad+\frac{K_1^2}{\kappa}\left(r_0\|\xi-\eta\|_2^2+(T+1)\frac{\e^{2K(T-r_0)}-1}{2K}|\xi(0)-\eta(0)|^2\right).\end{split}\end{equation*}

\smallskip\noindent\textup{ii)} \
For any $p>1$, $\xi, \eta\in \C$ and
non-negative $f\in \B_b(\C)$,
\beg{equation*}\beg{split} (P_Tf)^p(\eta)& \le
P_Tf^p(\xi)\left(\E\exp\left[\frac{p}{(p-1)^2}\,|\xi(0)-\eta(0)|^2
\left(\int_0^{T-r_0}\e^{-2Kt}\,\d S(t)\right)^{-1}\right]
\right)^{p-1}\\
&\quad\times\exp\left[\frac{p}{p-1} \frac{K_1^2}{\kappa}\left(r_0\|\xi-\eta\|_2^2
+(T+1)\frac{\e^{2K(T-r_0)}-1}{2K}\,|\xi(0)-\eta(0)|^2\right)\right].
 \end{split}\end{equation*}
 \end{thm}
 \begin{rem}\label{nd} If $B=0$, then we can
 choose $r_0=0$ and $K_1=0$, and thus the assertions in Theorem \ref{T3.2} reduce to the
 ones derived in \cite{WW} for the case without delay.
 \end{rem}

For a measurable space $(E,\scr F)$, let $\scr P(E)$ denote
 the family of all probability measures on $(E,\F)$.
 For $\mu,\nu\in\scr P(E)$, the entropy $\Ent(\nu|\mu)$
 is defined by
 $$\Ent(\nu|\mu):= \beg{cases} \int (\log \ff{\d\nu}{\d\mu})\,\d\nu, \ &\text{if}\ \nu\ \text{ is\ absolutely\ continuous\ with\ respect\ to}\ \mu,\\
 \infty,\ &\text{otherwise;}\end{cases}$$
 the total variation distance
 $\|\mu-\nu\|_{\operatorname{var}}$ is defined by
$$\|\mu-\nu\|_{\operatorname{var}} := \sup_{A\in\F}|\mu(A)-\nu(A)|.$$ By Pinsker's inequality (see \cite{CK, Pin}),
\beq\label{ETX} \|\mu-\nu\|_{\operatorname{var}}^2\le \ff 1 2 \Ent(\nu|\mu),\quad \mu,\nu\in \scr P(E).\end{equation}

For $\xi\in\C$, let $P_T(\xi,\cdot)$ be 
the distribution of $X_T^\xi$. 
The following corollary is a direct consequence
of Theorem \ref{T3.2},
see \cite[Theorem 1.4.2]{Wbook} for the proof; we also refer
to \cite[Subsection 1.4.1]{Wbook} for an in-depth
explanation of the applications of the Harnack inequalities.

\begin{cor}\label{density0} Let the assumptions in Theorem \ref{T3.2} hold.
 Then the following assertions hold.

\smallskip\noindent\textup{i)} \
For any $\xi,\eta\in\C$, $P_T(\xi,\cdot)$ is equivalent to $P_T(\eta,\cdot)$ and
\begin{align*}\Ent\big(P_{T}(\xi,\cdot)|P_{T}(\eta,\cdot)\big)
&\leq |\xi(0)-\eta(0)|^2\mathbb{E}\left(\int_0^{T-r_0}\e^{-2Kt}\,\d S(t)\right)^{-1}\\
&\quad+\frac{K_1^2}{\kappa}
\left(r_0\|\xi-\eta\|_2^2+(T+1)\frac{\e^{2K(T-r_0)}-1}{2K}\,
|\xi(0)-\eta(0)|^2\right),
\end{align*}
which together with Pinsker's inequality \eqref{ETX} implies that
\begin{align*}
2\|P_T(\xi,\cdot)-P_T(\eta,\cdot)\|_{\operatorname{var}}^2
&\le  |\xi(0)-\eta(0)|^2\mathbb{E}\left(\int_0^{T-r_0}\e^{-2Kt}\,\d S(t)\right)^{-1}\\
&\quad
+\frac{K_1^2}{\kappa}\left(r_0\|\xi-\eta\|_2^2+(T+1)\frac{\e^{2K(T-r_0)}-1}{2K}|\xi(0)-\eta(0)|^2\right). \end{align*}

\smallskip\noindent\textup{ii)} \
For any $p>1$ and $\xi,\eta\in\C$,
\begin{align*}&P_T\left\{\left(\frac{\d P_T(\xi,\cdot)}{\d P_T(\eta,\cdot)}\right)^{1/(p-1)}\right\}(\xi)\leq  \E\exp\left[\frac{p}{(p-1)^2}\,|\xi(0)-\eta(0)|^2\left(\int_0^{T-r_0}\e^{-2Kt}\,\d S(t)\right)^{-1}\right]\\
&\qquad\qquad\quad\times\exp\left[\frac{p}{(p-1)^2} \frac{K_1^2}{\kappa}\left(r_0\|\xi-\eta\|_2^2
+(T+1)\frac{\e^{2K(T-r_0)}-1}{2K}\,|\xi(0)-\eta(0)|^2\right)\right].
\end{align*}
\end{cor}

\begin{exa}\label{ex1}
Assume that {\bf(H)} holds with $K=0$.
Let $T> r_0$, and $S$ be a subordinator
  with Bernstein function $\phi(u)\geq\kappa u+cu^\alpha$
    \textup{(}$\kappa\geq0$, $c>0$, $0<\alpha<1$\textup{)}.

    \smallskip\noindent\textup{i)} \
    There exists $C=C(\alpha,c)>0$ such that
    for any $\xi, \eta\in \C$ and
$f\in \B_b(\C)$ with $f\geq1$,
\begin{align*}
P_T\log f(\eta)&\leq\log P_T f(\xi)+
\frac{C|\xi(0)-\eta(0)|^2}{[\kappa(T-r_0)]
\vee(T-r_0)^{1/\alpha}}\\
 &\quad+\frac{K_1^2}{\kappa}\left(r_0\|\xi-\eta\|_2^2+(T+1)(T-r_0)|\xi(0)-\eta(0)|^2\right).
\end{align*}

\smallskip\noindent\textup{ii)} \
If in addition $1/2<\alpha<1$,
then there exists $C=C(\alpha,c)>0$
such that for any $p>1$, $\xi, \eta\in \C$ and
non-negative $f\in \B_b(\C)$,
\beg{equation*}\beg{split}
&(P_Tf)^p(\eta) \le
P_Tf^p(\xi)\cdot\exp\left[\frac{p}{p-1} \frac{K_1^2}{\kappa}\left(r_0\|\xi-\eta\|_2^2
+(T+1)(T-r_0)|\xi(0)-\eta(0)|^2\right)\right]\\
&\quad\times
\exp\left[
C\left(
\frac{p|\xi(0)-\eta(0)|^2}{(p-1)(T-r_0)^{1/\alpha}}
+\frac{\left[p|\xi(0)-\eta(0)|^2\right]^{\alpha/(2\alpha-1)}}
{\left[(p-1)(T-r_0)\right]^{1/(2\alpha-1)}}
\right)
\wedge
\frac{p|\xi(0)-\eta(0)|^2}{(p-1)\kappa(T-r_0)}
\right].
 \end{split}\end{equation*}
\end{exa}

\section{Harnack inequalities under deterministic time-change}

Let $\ell:[0,\infty)\rightarrow[0,\infty)$ be
a sample path of $S$ (with Bernstein function $\phi$ of
the form \eqref{bern}), which is a non-decreasing and
c\`{a}dl\`{a}g function with $\ell(0)=0$.
By {\bf(H)} and the same explanation as in Remark \ref{EAU}, for any $\xi\in\C$, the following functional SDE has a unique non-explosive
solution with $X_0^{\ell}=\xi$:
\begin{equation}\label{jg4dgv}
\d X^\ell(t)=b(X^\ell(t))\,\d t+B(X^\ell_t)\,\d t +\d W(\ell(t)).
\end{equation}
We denote the solution by $X_t^{\ell,\xi}$. Let
$$P^\ell_t f(\xi)=\mathbb{E}f(X_t^{\ell,\xi}),\quad t\geq0,f\in\B_b(\C),\xi\in\C.$$

\begin{prp}\label{dfr3s}
Assume {\bf(H)} and  let $T>r_0$.

\smallskip\noindent\textup{i)} \
For any $\xi, \eta\in \C$ and
$f\in \B_b(\C)$ with $f\geq1$,
\begin{align*}
P_T^{\ell} \log f(\eta)
&\leq \log P_T^{\ell} f(\xi)+
|\xi(0)-\eta(0)|^2\left(
\int_0^{T-r_0}\e^{-2Kt}\,\d \ell(t)
\right)^{-1}\\
&\quad+\frac{K_1^2}{\kappa}\left(r_0\|\xi-\eta\|_2^2+(T+1)\frac{\e^{2K(T-r_0)}-1}{2K}|\xi(0)-\eta(0)|^2\right).
\end{align*}

\smallskip\noindent\textup{ii)} \
For any $p>1$, $\xi, \eta\in \C$ and
non-negative $f\in \B_b(\C)$,
\begin{align*}
\left(P_T^{\ell}  f(\eta)\right)^p&
\leq P_T^{\ell} f^p(\xi)
\exp\left[\frac{p}{p-1}|\xi(0)-\eta(0)|^2\left(
\int_0^{T-r_0}\e^{-2Kt}\,\d \ell(t)
\right)^{-1}\right]\\
&\quad\times\exp\left[\frac{p}{p-1} \frac{K_1^2}{\kappa}\left(r_0\|\xi-\eta\|_2^2
+(T+1)\frac{\e^{2K(T-r_0)}-1}{2K}|\xi(0)-\eta(0)|^2\right)\right].
\end{align*}
\end{prp}

Following the line of \cite{Den14, DS16, WW, WZ15, Zha13},
for $\varepsilon\in(0,1)$, consider the
following regularization of $\ell$:
$$\ell^\varepsilon(t):=\frac{1}{\varepsilon}
\int_{t}^{t+\varepsilon}\ell(s)\,\d s+\varepsilon t
=\int_0^1\ell(\varepsilon s+t)\,\d s+\varepsilon t,
\quad t\geq0.$$
It is clear that, for each $\varepsilon\in(0,1)$,
the function $\ell^\varepsilon$ is absolutely
continuous, strictly increasing and satisfies
for any $t\geq0$
\begin{equation}\label{approximation}
    \ell^\varepsilon(t)\downarrow\ell(t)\quad
    \text{as $\varepsilon\downarrow0$}.
\end{equation}
For $\xi\in\C$, let $X_t^{\ell^\varepsilon,\xi}$ be the solution to the following functional SDE with initial value $\xi$:
$$
\d X^{\ell^\varepsilon,\xi}(t)=b(X^{\ell^\varepsilon,\xi}(t))\,\d t+B(X_t^{\ell^\varepsilon,\xi})\,\d t+\d W(\ell^\varepsilon(t)-\ell^\varepsilon(0)).
$$
The associated semigroup is
denoted by $P_t^{\ell^\varepsilon}$.
Note that this SDE is indeed driven by Brownian motions and
thus the method of coupling and Girsanov's transformation
can be used to establish the dimension-free
Harnack inequalities for $P_t^{\ell^\varepsilon}$.

\begin{lem}\label{Ptl}
Fix $\varepsilon\in(0,1)$, assume {\bf(H)} and  let $T>r_0$.

\smallskip\noindent\textup{i)} \
For any $\xi, \eta\in \C$ and
$f\in \B_b(\C)$ with $f\geq1$,
\begin{align*}
P_T^{\ell^\varepsilon} \log f(\eta)
&\leq \log P_T^{\ell^\varepsilon} f(\xi)+
|\xi(0)-\eta(0)|^2\left(
\int_0^{T-r_0}\e^{-2Kt}\,\d \ell^\varepsilon(t)
\right)^{-1}\\
&\quad+\frac{K_1^2}{\kappa}\left(r_0\|\xi-\eta\|_2^2+(T+1)\frac{\e^{2K(T-r_0)}-1}{2K}|\xi(0)-\eta(0)|^2\right).
\end{align*}

\smallskip\noindent\textup{ii)} \
For any $p>1$, $\xi, \eta\in \C$ and
non-negative $f\in \B_b(\C)$,
\begin{align*}
\left(P_T^{\ell^\varepsilon}  f(\eta)\right)^p&
\leq P_T^{\ell^\varepsilon} f^p(\xi)
\exp\left[\frac{p}{p-1}|\,\xi(0)-\eta(0)|^2\left(
\int_0^{T-r_0}\e^{-2Kt}\,\d \ell^\varepsilon(t)
\right)^{-1}\right]\\
&\quad\times\exp\left[\frac{p}{p-1} \frac{K_1^2}{\kappa}\left(r_0\|\xi-\eta\|_2^2
+(T+1)\frac{\e^{2K(T-r_0)}-1}{2K}
|\xi(0)-\eta(0)|^2\right)\right].
\end{align*}
\end{lem}
\begin{proof}
Due to the existence of the delay part $B$, we will construct couplings as follows.
Let $Y_t$ solve the equation
\begin{equation}\begin{split}\label{EY}
\d Y(t)&=b(Y(t))\,\d t+B(X_t^{\ell^\varepsilon,\xi})\,\d t\\
&\quad+\lambda(t)\I_{[0,\tau)}(t) \frac{X^{\ell^\varepsilon,\xi}(t)-Y(t)}{|X^{\ell^\varepsilon,\xi}(t)
-Y(t)|}|\xi(0)-\eta(0)|\,\d \ell^\varepsilon(t)+\d W(\ell^\varepsilon(t)-
\ell^\varepsilon(0))
\end{split}\end{equation}
with $Y_0=\eta$, where
$$\lambda(t):=\frac{\e^{-Kt}}{\int_0^{T-r_0}\e^{-2Ks}
\,\d \ell^\varepsilon(s)},\quad t\geq 0,$$
and
$$\tau:=T\wedge\inf\{t\geq 0\,;\, X^{\ell^\varepsilon,\xi}(t)=Y(t)\}$$
is the coupling time. It is clear that $(X^{\ell^\varepsilon,\xi}(t),Y(t))$ is well defined for $t<\tau$. By {\bf(H)}, we have
$$
\d |X^{\ell^\varepsilon,\xi}(t)-Y(t)|\leq K|X^{\ell^\varepsilon,\xi}(t)-Y(t)|\,\d t-\lambda(t)|\xi(0)-\eta(0)|\,\d \ell^\varepsilon(t),\quad t\in[0,\tau).
$$
Thus, for $t\in[0,\tau)$,
\begin{equation}\begin{split}\label{EX-Y'}
|X^{\ell^\varepsilon,\xi}(t)-Y(t)|&\leq \e^{Kt}|\xi(0)-\eta(0)|\left(1-\int_0^t\e^{-Ks}\lambda(s)\,\d \ell^\varepsilon(s)\right)\\
&\leq  \frac{\e^{Kt}\int_t^{T-r_0}\e^{-2Ks}\,\d \ell^\varepsilon(s)}{\int_0^{T-r_0}\e^{-2Ks}\,\d \ell^\varepsilon(s)}\,|\xi(0)-\eta(0)|\\
&=:\Gamma(t)|\xi(0)-\eta(0)|.
\end{split}\end{equation}
If $\tau(\omega)>T-r_0$ for some $\omega\in\Omega$, we can
take $t=T-r_0$ in the above inequality to get
$$
    0<|X^{\ell^\varepsilon,\xi}(t)(\omega)-Y(t)(\omega)|
    \leq0,
$$
which is absurd. Therefore, $\tau\leq T-r_0$. Letting $Y(t)=X^{\ell^\varepsilon,\xi}(t)$ for $t\in[\tau,T]$, $Y(t)$ solves \eqref{EY} for $t\in[\tau,T]$. In particular, $X^{\ell^\varepsilon,\xi}_T=Y_T$. Moreover, by \eqref{EX-Y'} and $\tau\leq T-r_0$, we have \begin{align}\label{r-0}|X^{\ell^\varepsilon,\xi}(t)-Y(t)|^2\leq |\xi(0)-\eta(0)|^2\Gamma(t)^2\I_{[0,T-r_0]}(t), \ \ t\in[0,T].
\end{align}

Denote by $\gamma^\varepsilon:[\ell^\varepsilon(0),\infty)
\rightarrow[0,\infty)$ the inverse function of $\ell^\varepsilon$. Then $\ell^\varepsilon(\gamma^\varepsilon
(t))=t$ for $t\geq\ell^\varepsilon(0)$,
$\gamma^\varepsilon(\ell^\varepsilon(t))=t$
for $t\geq0$, and $t\mapsto\gamma^\varepsilon(t)$
is absolutely continuous and
strictly increasing. Let
$$
\widetilde{W}_t:=\int_{0}^{t}\Psi(u)\,\d u+
W(t)\quad\text{and}\quad M_t:=-\int_0^t
\< \Psi(u), \d
W(u)\>,\quad t\geq0,
$$
where $\Psi(u):=\Phi\circ
\gamma^\varepsilon
(u+\ell^\varepsilon(0))$ and
$$
\Phi(u):=[B(X_u
^{\ell^\varepsilon,\xi})-
B(Y_u)]
\frac{1}{(\ell^\varepsilon)'(u)}
+\lambda(u)
\I_{[0,\tau)}(u) \frac{X^{\ell^\varepsilon,\xi}(u))-Y(u)}
{|X^{\ell^\varepsilon,\xi}(u)
-Y(u)|}
\,|\xi(0)-\eta(0)|.
$$
By {\bf(H)}, the compensator of the martingale $M_t$
satisfies, for $t\geq0$,
\begin{equation}\begin{split}\label{ddfaw}
    \<M\>_t&=\int_0^t|\Psi(u)|^2\,\d u
    \leq\int_0^T|\Phi(s)|^2\,\d\ell^\varepsilon(s)\\
    &\leq 2K_1^2\int_0^T\|X_t^{\ell^\varepsilon,\xi}
    -Y_t\|^2_{2}\frac{1}{(\ell^\varepsilon)'(t)}\,\d t+2|\xi(0)-\eta(0)|^2\int_0^{T-r_0}|\lambda(t)|^2\,\d \ell^\varepsilon(t).
\end{split}\end{equation}
Recalling that $\ell$ is a sample path of
the subordinator $S$ with drift
parameter $\kappa\geq0$, one has
$$(\ell^\varepsilon)'(t)=
\frac{\ell(t+\varepsilon)-\ell(t)}{\varepsilon}+\varepsilon>\kappa,$$
and therefore
\begin{equation}\label{j3dc34d}
\int_0^T\|X_t^{\ell^\varepsilon,\xi}-Y_t\|^2_{2}\frac{1}{(\ell^\varepsilon)'(t)}\,\d t\leq\frac{1}{\kappa}\int_0^T\|X_t^{\ell^\varepsilon,\xi}
-Y_t\|^2_{2}\,\d t.
\end{equation}
Next, we focus on the estimate of $\int_0^T\|X_t^{\ell^\varepsilon,\xi}-Y_t\|^2_{2}\,\d t$.
Firstly, it is clear that for any $t\in[0,T]$,
\begin{align*}
\int_{-r_0}^0|X^{\ell^\varepsilon,\xi}_t(s)-Y_t(s)|^2\,\d s&=\int_{-r_0}^0|X^{\ell^\varepsilon,\xi}(t+s)-Y(t+s)|^2\,\d s\\
&=\int_{t-r_0}^t|X^{\ell^\varepsilon,\xi}(s)-Y(s)|^2\,\d s.
\end{align*}
This implies that for $t\in[0,r_0]$,
\begin{equation}\begin{split}\label{EX-Y''}
\int_{-r_0}^0|X^{\ell^\varepsilon,\xi}_t(s)-Y_t(s)|^2\,\d s
&=\left(\int_{t-r_0}^0+\int_0^t\right)
|X^{\ell^\varepsilon,\xi}(s)-Y(s)|^2\,\d s\\
&\leq \left(\int_{-r_0}^0+\int_{0}^{r_0}\right)
|X^{\ell^\varepsilon,\xi}(s)-Y(s)|^2\,\d s\\
&=\int_{-r_0}^0|\xi(s)-\eta(s)|^2\,\d s+\int_{0}^{r_0}|X^{\ell^\varepsilon,\xi}(s)-Y(s)|^2\,\d s,
\end{split}\end{equation}
and by \eqref{r-0}, for $t\in[r_0,T]$,
\begin{equation}\label{EX-Y'''}
\int_{-r_0}^0|X^{\ell^\varepsilon,\xi}_t(s)-Y_t(s)|^2\,\d s
\leq \int_{0}^{T-r_0}|X^{\ell^\varepsilon,\xi}
(s)-Y(s)|^2\,\d s.
\end{equation}
Combining \eqref{r-0}, \eqref{EX-Y''} and \eqref{EX-Y'''}, we
obtain
\begin{equation}\begin{split}\label{EX-Y3}
&\int_0^T\|X_t^{\ell^\varepsilon,\xi}-Y_t\|^2_{2}\,\d t\\
&\qquad=\int_0^{r_0}\left(\int_{-r_0}^0|X^{\ell^\varepsilon,\xi}
_t(s)-Y_t(s)|^2\,\d s\right)\,\d t+\int_{r_0}^T\left(\int_{-r_0}^0|X^{\ell^\varepsilon,\xi}_t(s)-Y_t(s)|^2\,\d s\right)\,\d t\\
&\qquad\quad+\int_0^T|X^{\ell^\varepsilon,\xi}
(t)-Y(t)|^2\,\d t\\
&\qquad\leq r_0\left(\int_{-r_0}^0|\xi(s)-\eta(s)|^2\,\d s+\int_{0}^{ r_0}|X^{\ell^\varepsilon,\xi}(s)-Y(s)|^2
\,\d s\right)\\
&\qquad\quad+(T-r_0)\left(\int_{0}^{T-r_0}
|X^{\ell^\varepsilon,\xi}(s)-Y(s)|^2\,\d s\right)+\int_0^{T-r_0}|X^{\ell^\varepsilon,\xi}
(t)-Y(t)|^2\,\d t\\
&\qquad\leq r_0\|\xi-\eta\|_2^2+(T+1)|\xi(0)-\eta(0)|^2
\int_{0}^{T-r_0}\Gamma(s)^2\,\d s\\
&\qquad\leq r_0\|\xi-\eta\|_2^2+(T+1)\frac{\e^{2K(T-r_0)}-1}{2K}\,|\xi(0)-\eta(0)|^2,
\end{split}\end{equation}
where in the last inequality we have used $\Gamma(s)\leq \e^{Ks}$
for $s\in[0,T-r_0]$.
By the definition of $\lambda(t)$, it is easy to see that
$$
2|\xi(0)-\eta(0)|^2\int_0^{T-r_0}|\lambda(t)|^2\,\d \ell^\varepsilon(t)\leq 2|\xi(0)-\eta(0)|^2\left(
\int_0^{T-r_0}\e^{-2Kt}\,\d \ell^\varepsilon(t)
\right)^{-1}.
$$
This, together with \eqref{ddfaw}, \eqref{j3dc34d}
and \eqref{EX-Y3}, yields that for any $t\geq0$
\begin{equation}\begin{split}\label{ddf23ds}
\<M\>_t&\leq \frac{2K_1^2}{\kappa}\left(r_0\|\xi-\eta\|_2^2+(T+1)\frac{\e^{2K(T-r_0)}-1}{2K}|\xi(0)-\eta(0)|^2\right)\\
&\quad+2|\xi(0)-\eta(0)|^2\left(
\int_0^{T-r_0}\e^{-2Kt}\,\d \ell^\varepsilon(t)
\right)^{-1}.
\end{split}\end{equation}

By Novikov's criterion, we have $\E R=1$, where
$$
R:=\exp\left[M_{\ell^\varepsilon(T)-\ell^\varepsilon(0)}
-\frac12\<M\>_{\ell^\varepsilon(T)-\ell^\varepsilon(0)}
\right].
$$
According to Girsanov's theorem,
$(\widetilde{W}_t)_{0\leq t\leq\ell^\varepsilon(T)-\ell^\varepsilon(0)}$ is a
$d$-dimensional Brownian motion under the new probability
measure $R\P$. Rewrite \eqref{EY} as
$$
    \d Y(t)=b(Y(t))\,\d t+B(Y_t)\,\d t
    +\d \widetilde{W}(\ell^\varepsilon(t)-
    \ell^\varepsilon(0)).
$$
Thus, the distribution of $(Y_t)_{0\leq t\leq T}$ under $R\P$
coincides with that of
$(X^{\ell^\varepsilon,\eta}_t)_{0\leq t\leq T}$
under $\P$; in particular, it holds that for
any $f\in \B_b(\C)$,
\begin{equation}\label{jh43cdf}
    \E f(X^{\ell^\varepsilon,\eta}_T)=
    \E_{R\P}f(Y_T)=
    \E\left[Rf(Y_T)\right]
    =\E\big[Rf(X^{\ell^\varepsilon,\xi}_T)\big].
\end{equation}
By \eqref{jh43cdf}, the Young
inequality (cf. \cite[p.\ 24]{Wbook}),
and the observation that
\begin{align*}
    \log R&=-\int_0^{\ell^\varepsilon(T)-
    \ell^\varepsilon(0)}\<\Psi(u),\d W(u)\>
    -\frac12\int_0^{\ell^\varepsilon(T)-
    \ell^\varepsilon(0)}|\Psi(u)|^2\,\d u\\
    &=-\int_0^{\ell^\varepsilon(T)-
    \ell^\varepsilon(0)}\<\Psi(u),\d \widetilde{W}(u)\>
    +\frac12\<M\>_{\ell^\varepsilon(T)-
    \ell^\varepsilon(0)},
\end{align*}
we get that, for any $f\in \B_b(\C)$ with $f\geq1$,
\begin{align*}
P_T^{\ell^\varepsilon}\log f(\eta)
&=\E\log f(X^{\ell^\varepsilon,\eta}_T)\\
&=\E\big[R\log f(X^{\ell^\varepsilon,\xi}_T)\big]\\
&\leq\log\E f(X^{\ell^\varepsilon,\xi}_T)
+\E[R\log R]\\
&=\log P_T^{\ell^\varepsilon}f(\xi)+
\E_{R\P}\log R\\
&=\log P_T^{\ell^\varepsilon}f(\xi)+
\frac12\<M\>_{\ell^\varepsilon(T)-
    \ell^\varepsilon(0)}.
\end{align*}
Combining this with \eqref{ddf23ds}, we obtain the
desired log-Harnack inequality.

Next, we prove the second assertion of the theorem.
For any non-negative $f\in \B_b(\C)$, we find
with \eqref{jh43cdf} and the H\"{o}lder inequality
\begin{equation}\begin{split}\label{fds54fff}
    (P_T^{\ell^\varepsilon}f)^p(\eta)&=\big(\E
    f(X_T^{\ell^\varepsilon,\eta})\big)^p\\
    &=\big(\E\big[R
    f(X_T^{\ell^\varepsilon,\xi})\big]\big)^p\\
    &\leq P_T^{\ell^\varepsilon}f^p(\xi)
    \cdot\big(
    \E\big[R^{p/(p-1)}\big]
    \big)^{p-1}.
\end{split}\end{equation}
Since by \eqref{ddf23ds}
\begin{align*}
    R^{p/(p-1)}&=\exp\left[
    \frac{p}{p-1}M_{\ell^\varepsilon(T)-\ell^\varepsilon(0)}
    -\frac{p}{2(p-1)}
    \<M\>_{\ell^\varepsilon(T)-\ell^\varepsilon(0)}
    \right]\\
    &=\exp\left[
    \frac{p}{2(p-1)^2}\<M\>_{\ell^\varepsilon(T)-\ell^\varepsilon(0)}
    \right]\\
    &\quad\times
    \exp\left[
    \frac{p}{p-1}M_{\ell^\varepsilon(T)-\ell^\varepsilon(0)}
    -\frac{p^2}{2(p-1)^2}
    \<M\>_{\ell^\varepsilon(T)-\ell^\varepsilon(0)}
    \right]\\
    &\leq\exp\left[
    \frac{p}{(p-1)^2}|\xi(0)-\eta(0)|^2
    \left(
    \int_0^{T-r_0}\e^{-2Kt}\,\d \ell^\varepsilon(t)
    \right)^{-1}
    \right]\\
    &\quad\times\exp\left[
    \frac{pK_1^2}{(p-1)^2\kappa}\left(
    r_0\|\xi-\eta\|_2^2+(T+1)\frac{\e^{2K(T-r_0)}-1}{2K}|\xi(0)-\eta(0)|^2
    \right)
    \right]\\
    &\quad\times\exp\left[
    \frac{p}{p-1}M_{\ell^\varepsilon(T)-\ell^\varepsilon(0)}
    -\frac{p^2}{2(p-1)^2}
    \<M\>_{\ell^\varepsilon(T)-\ell^\varepsilon(0)}
    \right],
\end{align*}
and noting the fact that
$\exp\left[
\frac{p}{p-1}M_{\ell^\varepsilon(t)-\ell^\varepsilon(0)}
-\frac{p^2}{2(p-1)^2}
\<M\>_{\ell^\varepsilon(t)-\ell^\varepsilon(0)}
\right]$, $0\leq t\leq T$, is a martingale
with mean $1$ -- this is due
to Novikov's criterion -- we know that
\begin{align*}
    \E\left[R^{p/(p-1)}\right]
    &\leq\exp\left[
    \frac{p}{(p-1)^2}|\,\xi(0)-\eta(0)|^2
    \left(
    \int_0^{T-r_0}\e^{-2Kt}\,\d \ell^\varepsilon(t)
    \right)^{-1}
    \right]\\
    &\quad\times\exp\left[
    \frac{pK_1^2}{(p-1)^2\kappa}\left(
    r_0\|\xi-\eta\|_2^2+(T+1)\frac{\e^{2K(T-r_0)}-1}{2K}\,|\xi(0)-\eta(0)|^2
    \right)
    \right].
\end{align*}
Inserting this estimate into \eqref{fds54fff}, we
get the power-Harnack inequality.
\end{proof}

To prove Proposition \ref{dfr3s} by
using Lemma \ref{Ptl}, we first prove
the following lemma.

\begin{lem}\label{approxi} Let $\varepsilon\in(0,1)$
and $T>0$. If
 $g^{(\varepsilon)}:[-r_0,\infty)\rightarrow[0,\infty)$
satisfies $g^{(\varepsilon)}(s)=0$ for $s\in[-r_0,0]$,
$\int_0^T\|g^{(\varepsilon)}_t\|_2^2\,\d t<\infty$,
and
$$|g^{(\varepsilon)}(t)|^2\leq C\int_0^t
\|g^{(\varepsilon)}_r\|_2^2\,\d r
+h^{(\varepsilon)}(t),\quad t\in[0,T],$$
where $C>0$ is a constant and $h^{(\varepsilon)}:[0,T]\rightarrow[0,\infty)$ is
measurable such that
$$\sup_{\varepsilon\in(0,1),t\in[0,T]}h^{(\varepsilon)}(t)
<\infty$$ and $\lim_{\varepsilon\downarrow0}
h^{(\varepsilon)}(t)=0$ for any $t\in[0,T]$.
Then we have
$$\lim_{\varepsilon\downarrow0}\|g^{(\varepsilon)}_t\|_2=0
,\quad t\in[0,T].$$
\end{lem}

\begin{proof}
    Since $g^{(\varepsilon)}(s)=0$
    for $s\in[-r_0,0]$, it holds that
    \begin{align*}
\int_{-r_0}^0|g^{(\varepsilon)}(t+s)|^2\,\d s&=\left(\int_{-r_0+t}^0+\int_0^t\right)
|g^{(\varepsilon)}(s)
|^2\,\d s\\
&\leq\int_0^t|g^{(\varepsilon)}(s)|^2\,\d s\\
&\leq C\int_0^t\left(\int_0^s
\|g^{(\varepsilon)}_r\|_2^2\,\d r
\right)\d s
+\int_0^th^{(\varepsilon)}(s)
\,\d s\\
&\leq Ct\int_0^t
\|g^{(\varepsilon)}_r\|_2^2\,\d r
+\int_0^th^{(\varepsilon)}(s)
\,\d s.
\end{align*}
Thus, we find that for any $t\in[0,T]$
\begin{align*}
    \|g^{(\varepsilon)}_t\|_2^2
    &=\int_{-r_0}^0|g^{(\varepsilon)}(t+s)|^2\,\d s
+|g^{(\varepsilon)}(t)|^2\\
&\leq C(t+1)\int_0^t
\|g^{(\varepsilon)}_r\|_2^2\,\d r+H^{(\varepsilon)}(t)\\
&\leq C(T+1)\int_0^t
\|g^{(\varepsilon)}_r\|_2^2\,\d r+H^{(\varepsilon)}(t),
\end{align*}
where
$$
    H^{(\varepsilon)}(t):=
    h^{(\varepsilon)}(t)
+\int_0^th^{(\varepsilon)}(s)
\,\d s.
$$
Now we can apply Gronwall's inequality to get
that, for all $t\in[0,T]$,
$$
    \|g^{(\varepsilon)}_t\|_2^2
    \leq H^{(\varepsilon)}(t)
    +C(T+1)\int_0^tH^{(\varepsilon)}(s)\,
    \e^{C(T+1)(t-s)}\,\d s.
$$
By our assumptions, we know that $\lim_{\varepsilon\downarrow0}
H^{(\varepsilon)}(t)=0$ for all $t\in[0,T]$. Letting $\varepsilon\downarrow0$ on
both sides of the above inequality and using
the dominated convergence
theorem, we complete the proof.
\end{proof}

\begin{proof}[Proof of Proposition \ref{dfr3s}]
Fix $T>r_0$. By a standard approximation argument, we may and do
assume that $f\in C_b(\C)$.

\emph{Step 1:} First, we assume
that $b$ is globally Lipschitzian: there exists a constant
$C>0$ such that
$$
    |b(x)-b(y)|\leq C|x-y|,\quad x,y\in\R^d.
$$
By the Lipschitz continuity of $b$ and $B$, and
noting that $|X^{\ell^\varepsilon,\xi}(r)-X^{\ell,\xi}(r)|
\leq\|X^{\ell^\varepsilon,\xi}_r
-X^{\ell,\xi}_r\|_2$, we have for
$t\geq0$
\begin{align*}
|X^{\ell^\varepsilon,\xi}(t)-X^{\ell,\xi}(t)|&\leq C\int_0^t|X^{\ell^\varepsilon,\xi}(r)-X^{\ell,\xi}(r)|\,\d r+K_1\int_0^t\|X^{\ell^\varepsilon,\xi}_r
-X^{\ell,\xi}_r\|_2\,\d r\\
&\quad+|W(\ell^{\varepsilon}(t)-\ell^{\varepsilon}(0))
-W(\ell(t))|\\
&\leq(C+K_1)\int_0^t\|X^{\ell^\varepsilon,\xi}_r
-X^{\ell,\xi}_r\|_2\,\d r+|W(\ell^{\varepsilon}(t)-\ell^{\varepsilon}(0))
-W(\ell(t))|.
\end{align*}
By the elementary inequality
    $$
        (u+v)^2\leq2u^2+2v^2,\quad u,v\geq0,
    $$
    and the H\"{o}lder inequality, we get that
    for $t\in[0,T]$,
\begin{align*}
    |X^{\ell^\varepsilon,\xi}(t)-X^{\ell,\xi}(t)|^2
    &\leq2(C+K_1)^2t\int_0^t\|X^{\ell^\varepsilon,\xi}_r
-X^{\ell,\xi}_r\|_2^2\,\d r+2|W(\ell^{\varepsilon}(t)-\ell^{\varepsilon}(0))
-W(\ell(t))|^2\\
&\leq2(C+K_1)^2T\int_0^t\|X^{\ell^\varepsilon,\xi}_r
-X^{\ell,\xi}_r\|_2^2\,\d r+2|W(\ell^{\varepsilon}(t)-\ell^{\varepsilon}(0))
-W(\ell(t))|^2.
\end{align*}
Applying Lemma \ref{approxi} with
$g^{(\varepsilon)}(t)
=|X^{\ell^\varepsilon,\xi}(t)-X^{\ell,\xi}(t)|$ and
$h^{(\varepsilon)}(t)
=2|W(\ell^{\varepsilon}(t)-\ell^{\varepsilon}(0))
-W(\ell(t))|^2$, we conclude that $X^{\ell^\varepsilon,\xi}_T\to X^{\ell,\xi}_T$
in $\C$ as $\varepsilon\downarrow0$, and so
$$
    \lim_{\varepsilon\downarrow0}
    P_T^{\ell^\varepsilon}f=
    P_T^\ell f,
    \quad f\in C_b(\C).
$$
Since $\ell$ is of bounded variation, it is easy
to get from \eqref{approximation} that
$$
    \lim_{\varepsilon\downarrow0}
    \int_0^{T-r_0}\e^{-2Kt}\,\d \ell^\varepsilon(t)
    =\int_0^{T-r_0}\e^{-2Kt}\,\d \ell(t).
$$
Letting $\varepsilon\downarrow0$ in Lemma \ref{Ptl},
we obtain the desired inequalities.

\medskip\emph{Step 2:} For the general case, we shall
make use of the approximation argument proposed
in \cite[part (c) of proof of Theorem 2.1]{WW}.
Let
$$
    \tilde{b}(x):=b(x)-Kx,\quad x\in\R^d.
$$
Then $\tilde{b}$ satisfies the dissipative condition:
$$
    \<\tilde{b}(x)-\tilde{b}(y),x-y\>\leq0,
    \quad x,y\in\R^d,
$$
and it is easy to see that the mapping
$\operatorname{id}-\varepsilon\tilde{b}:\R^d\rightarrow\R^d$
is injective for any $\varepsilon>0$.
For $\varepsilon>0$, let $\tilde{b}^{(\varepsilon)}$
be the Yoshida approximation of $\tilde{b}$, i.e.
$$
    \tilde{b}^{(\varepsilon)}:=
    \frac{1}{\varepsilon}\left[
    \left(
    \operatorname{id}-\varepsilon\tilde{b}
    \right)^{-1}(x)-x
    \right],\quad x\in\R^d.
$$
Then $\tilde{b}^{(\varepsilon)}$ is dissipative and globally Lipschitzian, $|\tilde{b}^{(\varepsilon)}|\leq |\tilde{b}|$ and $\lim_{\varepsilon\downarrow0}\tilde{b}^{(\varepsilon)}=\tilde{b}$. Let $b^{(\varepsilon)}(x):=\tilde{b}^{(\varepsilon)}(x)+Kx$. Then $b^{(\varepsilon)}$ is also Lipschitzian and $$\<b^{(\varepsilon)}(x)-b^{(\varepsilon)}(y),x-y\>\leq K|x-y|^2.$$
Let $X^{\ell,(\varepsilon),\xi}_t$ solve the SDE \eqref{jg4dgv}
with $b$
replaced by $b^{(\varepsilon)}$ and
$X^{\ell,(\varepsilon),\xi}_0=\xi\in\C$. Denote
by $P_t^{\ell,(\varepsilon)}$ the associated
semigroup. Due to the first part of
the proof, the statements of Proposition \ref{dfr3s}
hold with $P_t^{\ell}$ replaced
by $P_t^{\ell,(\varepsilon)}$. If
\begin{equation}\label{kj64cf3dk}
    \lim_{\varepsilon\downarrow0}
    P_T^{\ell,(\varepsilon)}f=
    P_T^\ell f,
    \quad f\in C_b(\C),
\end{equation}
then we complete the proof by
applying Proposition \ref{dfr3s}
with $P_t^{\ell}$ replaced
by $P_t^{\ell,(\varepsilon)}$ and
letting $\varepsilon\downarrow0$. Indeed, noting that
\begin{align*}
 &\d |X^{\ell,(\varepsilon),\xi}(t)-X^{\ell,\xi}(t)|^2\\
 &\qquad\quad=2\<X^{\ell,(\varepsilon),\xi}(t)-X^{\ell,\xi}(t), b^{(\varepsilon)}(X^{\ell,(\varepsilon),\xi}(t))
 -b^{(\varepsilon)}(X^{\ell,\xi}(t))\>\,\d t\\
 &\qquad\qquad+2\<X^{\ell,(\varepsilon),\xi}(t)-X^{\ell,\xi}(t), b^{(\varepsilon)}(X^{\ell,\xi}(t))
 -b(X^{\ell,\xi}(t))\>\,\d t\\
 &\qquad\qquad+2\<X^{\ell,(\varepsilon),\xi}(t)-X^{\ell,\xi}(t), B(X^{\ell,(\varepsilon),\xi}_t)-B(X^{\ell,\xi}_t)\>\,\d t\\
 &\quad\qquad\leq (2K+1) |X^{\ell,(\varepsilon),\xi}(t)
 -X^{\ell,\xi}(t)|^2\,\d t
 +|b^{(\varepsilon)}(X^{\ell,\xi}(t))
 -b(X^{\ell,\xi}(t))|^2\,\d t\\
 &\qquad\qquad+2K_1\|X^{\ell,(\varepsilon),\xi}_t
 -X^{\ell,\xi}_t\|_2^2\,\d t\\
 &\quad\qquad\leq (2|K|+2K_1+1)\|X^{\ell,(\varepsilon),\xi}_t
 -X^{\ell,\xi}_t\|_2^2\,\d t
 +|b^{(\varepsilon)}(X^{\ell,\xi}(t))
 -b(X^{\ell,\xi}(t))|^2\,\d t,
 \end{align*}
one has for $t\in[0,T]$
\begin{align*}
    &|X^{\ell,(\varepsilon),\xi}(t)-X^{\ell,\xi}(t)|^2\\
    &\quad\leq(2|K|+2K_1+1)\int_0^t
    \|X^{\ell,(\varepsilon),\xi}_r
 -X^{\ell,\xi}_r\|_2^2\,\d r
 +\int_0^t|b^{(\varepsilon)}(X^{\ell,\xi}(r))
 -b(X^{\ell,\xi}(r))|^2\,\d r\\
   &\quad=(2|K|+2K_1+1)\int_0^t
    \|X^{\ell,(\varepsilon),\xi}_r
 -X^{\ell,\xi}_r\|_2^2\,\d r
 +\int_0^t|\tilde{b}^{(\varepsilon)}(X^{\ell,\xi}(r))
 -\tilde{b}(X^{\ell,\xi}(r))|^2\,\d r.
\end{align*}
Applying Lemma \ref{approxi} with
$g^{(\varepsilon)}(t)=
|X^{\ell,(\varepsilon),\xi}(t)-X^{\ell,\xi}(t)|$
and $h^{(\varepsilon)}(t)
=\int_0^t|\tilde{b}^{(\varepsilon)}(X^{\ell,\xi}(r))
 \\ -\tilde{b}(X^{\ell,\xi}(r))|^2\,\d r$, we
 find that $X^{\ell,(\varepsilon),\xi}_T\to X^{\ell,\xi}_T$
in $\C$ as $\varepsilon\downarrow0$, and
thus \eqref{kj64cf3dk} follows.
\end{proof}

\section{Proofs of Theorem \ref{T3.2} and Example \ref{ex1}}

\begin{proof}[Proof of Theorem \ref{T3.2}]
    Since the processes $W$ and $S$ are independent,
    we have
    \begin{equation}\label{jg21hhsd}
        P_{T}f(\cdot)=\E\left[P_{T}^{\ell}f(\cdot)
        \left|_{\ell=S}\right.
        \right],\quad f\in\B_b(\C).
    \end{equation}
    By the first assertion of Proposition \ref{dfr3s},
    for all
    $f\in\B_b(\C)$ with $f\geq1$,
    \begin{align*}
        P_T\log f(\eta)&=\E\left[
        P_T^\ell\log f(\eta)\left|_{\ell=S}\right.
        \right]\\
        &\leq\E\left[
        \log P_T^\ell f(\xi)\left|_{\ell=S}\right.
        \right]+|\xi(0)-\eta(0)|^2\,\E\left(
\int_0^{T-r_0}\e^{-2Kt}\,\d S(t)
\right)^{-1}\\
&\quad+\frac{K_1^2}{\kappa}\left(r_0\|\xi-\eta\|_2^2
+(T+1)\frac{\e^{2K(T-r_0)}-1}{2K}|\xi(0)-\eta(0)|^2\right),
\end{align*}
which, together with the Jensen inequality and
\eqref{jg21hhsd}, implies the log-Harnack inequality.
Analogously, by the second assertion of
Proposition \ref{dfr3s},
    for all non-negative
    $f\in\B_b(\C)$,
\begin{align*}
        P_Tf(\eta)&=\E\left[
        P_T^\ell f(\eta)\left|_{\ell=S}\right.
        \right]\\
        &\leq\left.\E\left[
        \big(P_T^\ell f^p(\xi)\big)^{1/p}
        \exp\left[\frac{1}{p-1}|\xi(0)-\eta(0)|^2\left(
\int_0^{T-r_0}\e^{-2Kt}\,\d \ell(t)
\right)^{-1}\right]\right|_{\ell=S}\right]\\
&\quad
\times\exp\left[\frac{1}{p-1} \frac{K_1^2}{\kappa}\left(r_0\|\xi-\eta\|_2^2
+(T+1)\frac{\e^{2K(T-r_0)}-1}{2K}|\xi(0)-\eta(0)|^2\right)\right].
\end{align*}
It remains to use the H\"{o}lder inequality
and \eqref{jg21hhsd} to derive the power-Harnack
inequality.
\end{proof}

\begin{proof}[Proof of Example \ref{ex1}]
    By the assumption, one has
    $$S(t)\geq\kappa t+\tilde{S}(t)\geq (\kappa t)\vee\tilde{S}(t),\quad t\geq0,$$
    where $\tilde{S}$ is an $\alpha$-stable
    subordinator with Bernstein function
    $\tilde{\phi}(u)=cu^\alpha$. Combining this
    with Theorem \ref{T3.2} and the moment
    estimates for subordinators in
    \cite[Theorem 3.8\,(a) and (b)]{DS15}, we
    get the desired estimates.
\end{proof}


\begin{thebibliography}{99}


\bibitem{BWY} J. Bao, F.-Y. Wang, C. Yuan, \emph{Derivative formula and Harnack inequality for degenerate functionals SDEs,} Stoch. Dyn. 13(2013), 943-951.

\bibitem{BWY13} J. Bao, F.-Y. Wang, C. Yuan,  \emph{Bismut formulae and applications for functional SPDEs,} Bull. Sci. Math. 137 (2013), 509-522.

\bibitem{CK} I. Csisz\'ar,  J. K\"orne, \emph{Information Theory: Coding Theorems for Discrete Memory-less Systems,}  Academic Press, New York, 1981.

\bibitem{Den14} C.-S. Deng,  \emph{Harnack  inequalities
for SDEs driven by subordinate Brownian motions,} J. Math.
Anal. Appl. 417(2014), 970--978.


\bibitem{DS15}
C.-S. Deng, R.L. Schilling,
\emph{On shift Harnack inequalities for subordinate semigroups and moment estimates for L\'{e}vy processes,}
Stochastic Process. Appl. 125(2015), 3851--3878.

\bibitem{DS16}
C.-S. Deng, R.L. Schilling,
\emph{Harnack inequalities for SDEs driven by time-changed
fractional Brownian motions,} Electron. J.
Probab. 22(2017), 1--23.




\bibitem{HW} X. Huang, F.-Y. Wang, \emph{Functional SPDE with multiplicative noise and Dini drift,}  Ann. Fac. Sci. Toulouse Math. 6(2017), 519--537.

\bibitem{HZ} X. Huang, S.-Q. Zhang, \emph{Mild solutions and Harnack inequality for functional stochastic partial differential equations with Dini drift,} 
    J. Theoret. Probab. 32(2019), 303--329.

\bibitem{Pin} M. S. Pinsker, \emph{ Information and Information Stability of Random Variables and Processes,} Holden-Day, San Francisco, 1964.



\bibitem{RW}
M. R\"{o}ckner, F.-Y. Wang,
\emph{Log-Harnack inequality for stochastic differential equations in Hilbert spaces and its consequences,}
Inf. Dim. Anal. Quantum Probab. Rel. Top. 13(2010),27--37.

\bibitem{SSV}
R.L. Schilling, R. Song, Z. Vondra\v{c}ek:
\emph{Bernstein Functions. Theory and
Applications (2nd Edn)}, De Gruyter, Studies in Mathematics 37, Berlin 2012.


\bibitem{SWY} J. Shao, F.-Y. Wang, C. Yuan,  \emph{Harnack inequalities for stochastic (functional) differential
equations with non-Lipschitzian coefficients,} 
Electron. J. Probab. 17(2012), 1--18.


 \bibitem{FYW0} F.-Y. Wang,  \emph{Logarithmic Sobolev inequalities on noncompact Riemannian manifolds,}  Probab. Theory Related Fields 109(1997), 417--424.


\bibitem{Wbook} F.-Y. Wang, \emph{Harnack Inequalities
for Stochastic Partial Differential Equations,} Springer, New York, 2013.

\bibitem{WW} J. Wang, F.-Y. Wang,  \emph{Harnack  inequalities for stochastic equations driven by L\'{e}vy
    noise,} J. Math. Anal. Appl. 410(2014), 513--523.


\bibitem{WZ15}
L. Wang, X. Zhang, \emph{Harnack inequalities for SDEs driven by cylindrical $\alpha$-stable processes,} Potential Anal.
42(2015), 657--669.

\bibitem{Zha13}
X. Zhang, \emph{Derivative formula and gradient estimates for SDEs driven by $\alpha$-stable processes,}
Stochastic Process. Appl. 123(2013), 1213--1228.

\end{thebibliography}
\end{document}